\documentclass[12pt,reqno]{amsart}
\usepackage{amsfonts,amsthm}
\usepackage{hyperref}
\input epsf.sty

\newtheorem{thm}{Theorem}[section]
\newtheorem{lem}[thm]{Lemma}

\newtheorem{cor}[thm]{Corollary}
\theoremstyle{definition}
\newtheorem{defi}[thm]{Definition}
\newtheorem{prop}[thm]{Proposition}
\newtheorem{ex}[thm]{Example}
\newtheorem{question}[thm]{Open Problem}
\numberwithin{equation}{section}

\newcommand{\R}{\ensuremath{\mathbb{R}}}

\newcommand{\N}{\ensuremath{\mathbb{N}}}

\newcommand{\remark}{\medskip\noindent\textbf{Remark:\ }}

\newcommand{\U}{{\cal U}}
\newcommand{\Lo}{{\cal L}}
\newcommand{\Hu}{{\overline H}}
\newcommand{\Hl}{{\underline H}}

\newcommand{\bo}{\partial\Omega}

\newcommand{\ep}{\varepsilon}

\newcommand{\eps}{\epsilon}
\newcommand{\cal}{\mathcal}

\begin{document}

\title[An invariance of $p$-harmonic functions
under boundary perturbations]{Tug-of-war with noise and an
invariance of $p$-harmonic functions under boundary perturbations}

\author[Sungwook Kim]{Sungwook Kim}
\address{Courant Institute of Mathematical Sciences, New York University, New York,
NY,10012} \email{sungwooknyu@gmail.com}

\begin{abstract}
\label{sec:abstract} In this paper, we provide new results about an
invariance of $p$-harmonic functions under boundary perturbations by
using tug-of-war with noise; a probabilistic interpretation of
$p$-harmonic functions introduced by Peres-Sheffield in \cite{ps}.
As a main result, when $E\subset\bo $ is countable and $f\in
C(\bo)$, we provide a necessary and sufficient condition for $E$ to
guarantee that $H_g=H_f$ whenever $g=f$ on $\bo\setminus E$. Here
$H_f$ and $H_g$ denote the Perron solutions of $f$ and $g$. It turns
out that $E$ should be of $p$-harmonic measure zero with respect to
$\Omega$. As a consequence, we analyze a structure of a countable
set of $p$-harmonic measure zero. In particular, we give some
results for the subadditivity of $p$-harmonic measures and an
invariance result for $p$-harmonic measures. In addition, the
results in this paper solve the problem regarding a perturbation
point Bj\"orn \cite{Bjorn} suggested for the case of unweighted
$\R^n$.

\end{abstract}

\maketitle

\section{Introduction}
\label{sec:intro} A function $u$ on a domain $\Omega$ is called
\emph{$p$-harmonic} in $\Omega$ (for $1 < p < \infty$) if it is a
weak solution to
\[ \Delta_pu := \mathrm{div}(|Du|^{p-2}Du) = 0~~ \mathrm{ in }~ \Omega, \]
(or as viscosity solutions--see either \cite{viscosty} or Section 1
in \cite{ps}). That is, $u$ is $p$-harmonic in $\Omega$ if and only
if it belongs to the Sobolev space $W^{1,p}_{\text{loc}}(\Omega)$
(i.e., $\nabla u \in L^p_{\text{loc}}(\Omega)$) and
\[\label{euler} \int_{\Omega} |\nabla u|^{p-2} ( \nabla u, \nabla \phi ) dx =
0 \] for every $\phi \in C^{\infty}_0(\Omega)$. $\Delta_p$ is called
the {\it $p$-Laplace operator} or {\it $p$-Laplacian}.

The Dirichlet problem for the $p$-Laplace equation involves finding
a $p$-harmonic extension $u$ to $\Omega$ of a boundary function $f$
defined on $\partial \Omega$;
\begin{equation}
\label{eqn:Dirichlet} \Delta_p u=0 ~{\textrm {in} }~ \Omega~~
{\textrm {and} }~~ u=f~ \textrm{on}~ \bo.
\end{equation}
The existence and uniqueness of the solution for
(\ref{eqn:Dirichlet}) is well-known in the Sobolev sense. (See
\cite{hjk}.) However, due to non-linearity of the $p$-Laplacian,
there are many open problems. An intriguing problem is that of {\it
$p$-harmonic measure} which is the solution of (\ref{eqn:Dirichlet})
when $f=\chi_E$ and $E\subset \bo$. More precisely, the $p$-harmonic
measure of $E$ with respect to~$\Omega$ evaluated at $x\in \Omega$
is defined by
\[\omega_p(x;E,\Omega)=
\overline H_{\chi_E}(x)\] where $\Hu$ denotes the upper Perron
solution of (\ref{eqn:Dirichlet}). (See Section 2 for all the
definitions and notations.) It is well known that when $p=2$ and
$\Omega$ is regular, $\omega_p(x;\cdot,\Omega)$ defines a
probability measure on $\partial\Omega$, but when $p\neq2$,
$p$-harmonic measure is not a measure. Very little is known about
the measure theoretic properties of $p$-harmonic measure. Martio
\cite{Martio} asked whether $p$-harmonic measure defines an outer
measure on zero-level set of $p$-harmonic measure, i.e. whether
$p$-harmonic measure is subadditive on the sets whose $p$-harmonic
measure is zero. Llorente-Manfredi-Wu \cite{ljm} negatively answered
to Martio's question; when $\Omega$ is the upper half plane, there
exist sets $A, B\subset\bo$ such that
$\omega_p(A,\Omega)=\omega_p(B,\Omega)=0$, $A\cup B=\bo=\R$ and
$|\R\setminus A|=|\R\setminus B|=0$ where $|\cdot|$ stands for
Lebesgue measure on $\R$. However, as far as the author is aware,
the following problem concerning $p$-harmonic measure still remains
unsolved.

\begin{question}
\label{ques:pharmonicmeasure} When $E,F\subset\bo$ are both compact
and $\omega_p(E,\Omega)=\omega_p(F,\Omega)=0$, is it
$$\omega_p(E\cup F,\Omega)=0?$$
\end{question}
\noindent Further questions and discussions on $p$-harmonic measures
can be found in \cite{hjk}, \cite{ba} and \cite{bj1}.

Another interesting problem for \eqref{eqn:Dirichlet} is a {\it
boundary perturbation problem}; when $f,g$ are two boundary
functions on $\bo$ such that $f=g$ except $E\subset \bo$, what
condition for $E$ implies $H_f=H_g$? (Here $H_f$ and $H_g$ denotes
the Perron solutions of $f$ and $g$.) When $\Omega\subset\R^n$ is
bounded and $1<p\leq n$, an important result is obtained by
Bj\"orn-Bj\"orn-Shanmugalingam \cite{bj2}; if $f\in C(\bo)$ and
$g=f$ on $\bo$ except a set of $p$-capacity zero, then $H_f=H_g$.
Note that when $\Omega\subset\R^n$ and $p>n$, there exists no set of
$p$-capacity zero. Therefore the methods in \cite{bj2} cannot be
applied when $p>n$. There has been little work done when $p>n$. Even
when $n=2$ and $p>2$, a seemingly simple question suggested by
Baernstein \cite{ba} has not been answered until the works of
Bj\"orn \cite{Bjorn} and Kim-Sheffield \cite{ks}; if
$\Omega=B(0,1)\subset\R^2$, $E$ is a finite union of open arcs on
$\bo$, $f=\chi_E$ and $g=\chi_{\overline E}$, then $H_f=H_g$. A
first result for a boundary perturbation problem when $n\geq 2$ and
$p>n$ is given by Bj\"orn \cite{Bjorn}, where he introduced the
notion of a perturbation point which is a simple version of a
boundary perturbation problem;

\begin{defi}
\label{def:perturbation} Let $\Omega\subset\R^n$ be a bounded
domain. $x_0\in\bo$ is called {\it a perturbation point (of
$\Omega$)}; whenever $f\in C(\bo)$ and $g$ is a bounded function on
$\bo$ such that $g=f$ on $\bo\setminus\{x_0\}$, we have
$$ H_f=H_g.$$
\end{defi}
Note that not every regular boundary point is a perturbation point
as the following example shows.

\begin{ex}
\label{ex:notperturabtion} Let $n<p<\infty$ and
$\Omega=B(0,1)\setminus\{0\}\subset\R^n$. Let $f=0$ and
$g=\chi_{\{0\}}$ on $\bo$. Then we can verify that $H_f=0$ and
$H_g(x)=1-|x|^{\frac{p-n}{p-1}} $. Therefore, $H_f\neq H_g$ and $0$
is not a perturbation point.
\end{ex}

As one major result in \cite{Bjorn}, Bj\"orn showed that an
{exterior ray point} is always a perturbation point and $H_f=H_g$
whenever $f\in C(\bo)$ and $g=f$ on $\bo$ except countable exterior
ray points. Most of the results in \cite{Bjorn} can be extended by
replacing an exterior ray point with any perturbation point. By
observing that $0$ in Example \ref{ex:notperturabtion} is an
isolated boundary point, Bj\"orn proposed the following problem in
\cite{Bjorn};

\noindent{\bf Bj\"orn's problem )} Is it true that any regular point
which is not isolated among the regular boundary points is a
perturbation point?

In this paper, we give several invariance results for $p$-harmonic
functions including an affirmative answer to Bj\"orn's problem by
using {\it tug-of-war with noise}; a probabilistic interpretation of
$p$-harmonic functions introduced by Peres-Sheffield in \cite{ps}.
The main result is Theorem \ref{thm:perturbationcountablecase},
which reveals a link between $p$-harmonic measure and a boundary
perturbation problem as well as analyzes the structure of a
countable set of $p$-harmonic measure zero and gives a necessary and
sufficient condition for a boundary perturbation problem when $f\in
C(\bo)$ and $E\subset \bo$ is countable. An interesting fact is that
when $E\subset\bo$ is a countable set, a boundary perturbation
problem and $\omega_p(E,\Omega)=0$ are local properties. As other
important consequences, Theorem \ref{thm:countablesubadditive} and
Theorem \ref{thm:countablepmeasureinvari} show that $p$-harmonic
measure is subadditive on $\{E\subset\bo: E~\text{is~a
countable~set~of~$p$-harmonic~measure zero} \}$ and a countable set
of $p$-harmonic measure zero does not affect the $p$-harmonic
measure of any closed set on $\bo$. Theorem
\ref{thm:gameperturbatoin} and Theorem \ref{thm:gameneg} will play a
vital role to obtain most of results. In particular, Theorem
\ref{thm:gameneg} answers Bj\"orn's problem affirmatively and shows
the locality of a perturbation point.  All the results are new when
$p>n$.

The outline of the paper is as follows. In Section 2, we give some
preliminary results for $p$-harmonic functions and Perron solutions.
In Section 3, we give a brief explanation of tug-of-war with noise
and characterize a perturbation point in terms of tug-of-war with
noise. In Section 4, we give a necessary and sufficient condition
for a perturbation point, thereby answering Bjo\"rn's question
affirmatively. As applications, in Section 5, we provide several
results for $p$-harmonic measures as well as a boundary perturbation
problem with a countable set. Finally, in Section 6, we give some
open problems concerning a boundary perturbation problem and
$p$-harmonic measures.

\section{Definitions and preliminary results }
\label{sec:def} The main reference for the results and notation in this section
is \cite{hjk}.
\begin{defi}
A {\it domain} $\Omega\subset\R^n$ is an open connected subset. When
there exists $B(0,R)\subset\R^n$ such that $\Omega\subset B(0,R)$,
we say that $\Omega$ is a {\it bounded domain}.
\end{defi}

First, we state some properties of $p$-harmonic functions which will
be used later in this paper.
\begin{thm}{\rm(Strong maximum principle)}
A nonconstant $p$-harmonic function in a domain $\Omega$ cannot
attain its supremum or infimum.
\end{thm}

\begin{thm}{\rm(Harnack's convergence theorem)}
\label{thm:Harnack} Suppose that $u_i, i=1,2,\dots$, is an
increasing sequence of $p$-harmonic functions in $\Omega$. Then the
function $u=\lim_{i\rightarrow \infty}u_i$ is either $p$-harmonic in
$\Omega$ or identically $+\infty$.
\end{thm}

\begin{defi}
A function $u:\Omega \rightarrow (-\infty,\infty]$ is called {\it
$p$-superharmonic} in $\Omega$ if (i) $u$ is lower semicontinuous in
$\Omega$, (ii) $u \not = \infty$ in $\Omega$, and (iii) for each
domain $D\subset\subset \Omega$, the following
    comparison principle holds: if $h \in
C(\overline{D})$ is $p$-harmonic in $D$ and $u \geq h$ on $\partial
D$, then $u \geq h$ in $D$. We say that $u$ is {\it $p$-subharmonic}
in $\Omega$ if $-u$ is $p$-superharmonic in $\Omega$.
\end{defi}

The following comparison principle will be used many times
throughout this paper.
\begin{thm}{\rm(Comparison Principle)} Suppose that $u$ is
 $p$-superharmonic and $v$ is $p$-subharmonic in $\Omega$. If
\[\mathrm{limsup}_{y\rightarrow x}v(y)\leq \mathrm{liminf}_{y\rightarrow x}u(y)\]
for all $x\in\partial\Omega$, and also for $x=\infty$ if $\Omega$ is
unbounded, (excluding the situation $\infty\leq\infty$ and
$-\infty\leq-\infty$), then $v\leq u$ in $\Omega$.
\end{thm}

\begin{defi}
Let $f:\bo\rightarrow [-\infty, \infty]$. The {\it upper class}
${\U}_{f}$ consists of all the functions $u$ such that (i) $u$ is
$p$-superharmonic in $\Omega$, (ii) $u$ is bounded below, and (iii)
$\liminf_{x\rightarrow y}u(x)
    \geq f(x)$ for all $y\in\partial\Omega$.
The {\it lower class} $\Lo_f$ is defined as $v\in\Lo_f$ if and only
if $-v\in\U_{-f}$.
\end{defi}

\begin{defi}
The {\it upper Perron solution}, $\Hu_f$ and {\it lower Perron
solution}, $\Hl_f$ are defined by
$$\Hu_f(x)=\inf\{u(x): u\in\U \}~\text{and}~\Hl_f(x)=\sup
\{v(x): v\in\Lo \}.$$
\end{defi}

Note that the comparison principle shows that $\Hl_f\leq\Hu_f$. We
list some basic properties of the Perron solutions.
\begin{prop}\hfill
\label{prop:semidecreasing}
\begin{itemize}
\item[i)] $\Hl_f$ and $\Hu_f$ are
$p$-harmonic in $\Omega$ unless they are not identically $\pm
\infty$.
\item[ii)] Let $f_j:\bo\rightarrow [-\infty,\infty)$ be a decreasing
sequence of upper semicontinuous functions and $f=\lim f_j$. Then
$\Hu_f=\lim_{j\rightarrow \infty}\Hu_{f_j}.$
\end{itemize}
\end{prop}

For the boundary continuity of the Perron solutions, we introduce
the notion of regularity.
\begin{defi} $x_0\in\bo$ is called a
{\it regular} point of $\Omega$, if
$$ \lim_{x\rightarrow x_0}\Hu_{f}(x)=f(x_0)$$
for each continuous function $f:\bo\rightarrow \R$. A point is {\it
irregular} if it is not regular. If all boundary points of $\Omega$
are regular, then $\Omega$ is called {\it regular}.
\end{defi}
A necessary and sufficient condition for regularity is
well-known.(See Chapter 6 in \cite{hjk}.) In particular, any
Lipschitz domain is regular and when $p>n$, any domain is regular.

It is natural to ask which one of the two Perron solutions $\Hu_f$
and $\Hl_f$ is the ``correct" solution to the Dirichlet problem. We
introduce the notion of resolutivity.
\begin{defi}
We say that $f$ is {\it resolutive} if $\underline H_f$ and
$\overline H_f$ agree. When $f$ is resolutive, we denote the Perron
solution by $H_f:=\underline H_f = \overline H_f$ and call it {\it
the $p$-harmonic extension of $f$} to $\Omega$.
\end{defi}
When $p=2$, it is known that all measurable functions are
resolutive. It is an open question whether all measurable functions
are resolutive for general $p$. However, the following result is
known for resolutivity. For more details see Chapter 9 in \cite{hjk}

\begin{thm}
\label{thm:semicontinuous-resolutive on the regular} Let $\Omega$ be
regular. If $f$ is bounded and lower(or upper) semicontinuous on
$\bo$, then $f$ is resolutive in $\Omega$.
\end{thm}
\noindent{\bf Remark:} Theorem \ref{thm:semicontinuous-resolutive on
the regular} and Theorem \ref{thm:bjornpcapcityzero} shows that any
bounded function which is continuous except a single point is
resolutive. Therefore, $H_g$ is well-defined in Definition
\ref{def:perturbation}.

Now let us define $p$-harmonic measure by the upper Perron solution.

\begin{defi}
The function $\omega_p(x,E,\Omega)=\Hu_{\chi_E}(x)=\inf \U_E$ is
called the $p$-harmonic measure of $E\subset\bo$ at $x\in\Omega$
with respect to $\Omega$. If $\omega_p(E,\Omega)=0$, we say that $E$
is {\it of $p$-harmonic measure zero}.
\end{defi}

\begin{prop}\hfill
\label{prop:phamrnoicmeasure}
\begin{itemize}
\item [i)] $0\leq\omega_p(x,E,\Omega)
\leq 1$. Furthermore, if $\omega_p(x,E,\Omega)=0$ at some
$x\in\Omega$, then $\omega_p(x,E,\Omega)\equiv 0$ in $\Omega$.

    \item [ii)] If $E_1\subset E_2\subset\bo$, then
     $\omega_p(x,E_1,\Omega)\leq \omega_p(x,E_2,\Omega)$.
    \item [iii)] If $E\subset\bo_1\cap\bo_2$
    and if $\Omega_1\subset\Omega_2$, then
$\omega_p(x,E,\Omega_1)\leq \omega_p(x,E,\Omega_2)$ in $\Omega_1$.

\end{itemize}
\end{prop}
To Open problem \ref{ques:pharmonicmeasure}, there is a partial
answer. (See Theorem 11.17 in \cite{hjk}.)
 \begin{thm}
\label{compact sets} Let $1<p<\infty$ and let $\Omega$ be regular.
If $E,F\subset\bo$ are closed sets of $p$-harmonic measure zero and
$E\cap F=\phi$, then $\omega_p(E\cup F,\Omega)=0$.
\end{thm}

Next we introduce a notion of $p$-capacity.
\begin{defi}
The {\it $p$-capacity} of $E$ is defined by
\[C_p(E):=\mathrm{inf} \int_{{\ R^n}}(|u|^p+|\nabla u|^p) \]
where the infimum is taken over all $u\in W^{1,p}({\R^n})$ such that
$u=1$ in a neighborhood of $E$. If $C_p(E)=0$, we say that $E$ is a
set {\it of $p$-capacity zero}.
\end{defi}

Here are some basic properties of $p$-capacity and see Chapter 2 in
\cite{hjk} for more properties.
\begin{prop}\hfill
\begin{itemize}
\item[i)]
\label{prop:onepointcapacity}A point of $\R^n$ is of $p$-capacity
zero if and only if $1<p\leq n$. In particular, when $p>n$, there
exists no nonempty set of $p$-capacity zero.
\item[ii)] $C_p(\sum_i E_i)\leq \sum_i C_p(E_i)$. In particular, when $1<p\leq n$,
 every countable set is of $p$-capacity zero.
\end{itemize}
\end{prop}
A set of $p$-capacity zero can be described in terms of $p$-harmonic
measure.
\begin{defi}
\label{def:absolutezero} We say that $E\subset\R^n$ is of {\it
absolute $p$-harmonic measure zero} if $\omega_p(E\cap
\bo,\Omega)=0$ for all bounded domains $\Omega\subset \R^n$.
\end{defi}

We state Theorem 11.15 in \cite{hjk}.
\begin{thm}
\label{thm:absoluteequalpcapacity} $E$ is of absolute $p$-harmonic
measure zero if and only if $E$ is of $p$-capacity zero.
\end{thm}

When $1<p\leq n$ and $E\subset\bo$ is of $p$-capacity zero,
Bj\"orn-Bj\"orn-Shanmugalingam showed the following result for a boundary
perturbation problem.
\begin{thm}{\rm (Bj\"orn-Bj\"orn-Shanmugalingam \cite{bj2})}
\label{thm:bjornpcapcityzero} Assume that $f\in C(\bo)$ and $g=f$ on
$\bo$ except a set of $p$-capacity zero. Then $g$ is resolutive and
$$H_g=H_f.$$
\end{thm}

\begin{cor}
\label{cor:plessthanninvariant} Let $1<p\leq n$ and let
$\Omega\subset\R^n$ be a bounded domain.  Every point on $\bo$ is a
perturbation point.
\end{cor}

\begin{defi}
We say that $x_0\in\bo$ is an {\it exterior ray point} if there is a
line segment, $\mathcal L$ such that $x_0\in \mathcal L$ and
$\mathcal L\subset \R^n\setminus \Omega$.
\end{defi}
 For instance, if
$\Omega=B(0,1)\setminus\{0< x <1\}\subset \R^n$, $0$ is an exterior
ray point.

\begin{thm}{\rm (Bj\"orn \cite{Bjorn})}
\label{thm:bjornexterior} Let $1<p<\infty$ and let $\Omega\subset
\R^n$ be a bounded domain. An exterior ray point is a perturbation
point.
\end{thm}

\begin{thm}{\rm (Bj\"orn \cite{Bjorn})}
\label{thm:pmeasureneq} Let $1<p<\infty$ and let $\Omega\subset\R^n$
be a bounded domain. Let $E\subset\bo$ be a countable set whose
elements are perturbation points of $\Omega$. If $f\in C(\bo)$ and
$g=f$ on $\bo\setminus E$, then $g$ is resolutive and
\begin{equation}
\label{eqn:bjorn} H_g=H_f.
\end{equation}
In particular, when $E$ consists of exterior ray points,
\eqref{eqn:bjorn} holds.
\end{thm}

Note that a major part in Theorem \ref{thm:pmeasureneq} is when
$p>n$. When $1<p\leq n$, Theorem \ref{thm:pmeasureneq} is just a
consequence of Theorem \ref{thm:bjornpcapcityzero} because the
$p$-capacity of a countable set is always zero. Also note that we
neither require $g$ to be bounded nor to be continuous on $\{x\in
\bo: g(x)=f(x)\}$.

\section{Tug-of-war with noise and game-perturbation points}
\label{sec:tug-of-war} When $p=2$, it is discovered by Kakutani
\cite{kakutani} that the Dirichlet problem can be solved in a
probabilistic way; $u(x)=\mathbb E_x(f(B_{\tau}))$ where $\mathbb
E_x$ stands for the expected value when a Brownian motion $B$ starts
at $x$ and runs until hitting time $\tau$ of $\bo$. However, when
$p\neq 2$, a probabilistic interpretation of $p$-harmonic functions
has remained unknown until recently Peres-Sheffield's works. (See
also \cite{Manfredi2}.) Their works were initiated to figure out the
behaviors of two-player random turn games like a random turn hex
\cite{Hex}. After some further research, they found that the value of a
two-player random turn game is related to the $\infty$-Laplace
equation, $\Delta_{\infty}{u}:=|\nabla
u|^{-2}\Sigma_{i,j}u_iu_{i,j}u_j=0$, and named the game {\it
tug-of-war} \cite{PSSW}. By noticing $\Delta_pu=|\nabla u|^{p-2}\{
\Delta u + (p-2)\Delta_{\infty}{u} \}$, they finally showed that a
variant of tug-of-war, called {\it tug-of-war with noise}, gives a
probabilistic solution to the Dirichlet problem
(\ref{eqn:Dirichlet}).

In this section, we give a quick summary of tug-of-war with noise
and apply it to characterize a perturbation point in a probabilistic
way.\\

\noindent{\bf Tug-of-war with noise} Let $\Omega\subset\R^n$ be
bounded. Let $\alpha = 1 + \sqrt{(n-1)/(p-1)}$ and let
$f:\bo\rightarrow\R$ be the terminal payoff function. The game is
played as follows: At the $k$th step, a fair coin is tossed, and the
winning player is allowed to make a move $v$ with $|v| \leq
\epsilon$. If $ \mathrm{dist}(x_{k-1},
\partial \Omega)> \alpha \epsilon$, then the moving player chooses
$v_k \in \R^n$ with $|v_k|\leq \epsilon$ and sets $x_k = x_{k-1}+v_k
+ z_k$ where $z_k$ is a random ``noise" vector whose law is the
uniform distribution on the sphere of radius
$|v_k|\sqrt{(n-1)/(p-1)}$ in the hyperplane orthogonal to $v_k$.
(Here we chose a simple noise vector. See
\cite{ps} for more details of a noise vector.) If
$\mathrm{dist}(x_{k-1},
\partial \Omega) \leq \alpha \epsilon$, then the moving player
chooses an $x_k \in \partial \Omega$ with $|x_k - x_{k-1}| \leq
\alpha \epsilon$ and the game ends, with player I receiving a payoff
of $f(x_k)$ from player II. Both players receive a payoff of zero if
the game never terminates.

\begin{defi}
A {\it strategy} for players is a way of choosing the player's next
move as a function of all previously played moves and all coin
tosses. More precisely it is a sequence of Borel-measurable maps
from $\Omega\times(\overline {B(0,\ep)}\times \overline \Omega)^k$
to $\overline {B(0,\ep)}$, giving the move a player would make at
the $k$th step of the game as a function of the game history.
\end{defi}
Note that a pair of strategies $\sigma=(S_I, S_{II})$ (where $S_I$
is a strategy for player I and $S_{II}$ is a strategy for player II)
and a starting point $x$ determine a unique probability measure
$\mathbb P_x$ on the space of game position sequences. Let us denote
the corresponding expectation by $\mathbb E_x$.

\begin{defi}
The {\it value of the game for player I} at $x$ is defined by
$u_1^\eps(x) =\sup_{S_I}$ $\inf_{S_{II}} V_x (S_I, S_{II}) $ and the
{\it value of the game for player II} at $x$ is defined by
$u_2^\eps(x) = \inf_{S_{II}} \sup_{S_I} V_x (S_I, S_{II}) $ where
 $V_x (S_I, S_{II}) = \mathbb{E}_x
\Bigl[f(x_\tau) \mathbf{ \chi}_{\{\tau<\infty\}}\Bigr]$ is the
expected payoff and $\tau$ is the exit time of $\Omega$.
\end{defi}
By definitions, we always have $u_1^\eps(x)\leq u_2^\eps(x)$.

\begin{defi}
\label{def:gameregular}
 $x_0\in\bo$ is called a {\it game-regular} point of
 $\Omega$ if for every $\delta >0$
and $\eta >0$ there exists a $\delta_0$ and $\epsilon_0$ such that
for every $x\in \Omega\cap B(x_0,\delta_0)$ and $\epsilon <
\epsilon_0$, player I has a strategy that guarantees that an
$\epsilon$-step game started at $x$ will terminate at a point on
$\partial \Omega \cap B(x_0,\delta)$ with probability at least $1 -
\eta$. $\Omega$ is {\it game-regular} if every $x \in
\partial \Omega$ is game-regular.
\end{defi}

The main results in \cite{ps} are the followings.

\begin{thm}{\rm(Peres-Sheffield \cite{ps})}
\label{thm:sufficientconditiongameregular} Let $1<p<\infty$ and let
$\Omega$ be a bounded domain in $\R^n$.
\begin{itemize}
\item[i)] If $p
> n$, then $\Omega$ is game-regular.
\item[ ii)] If $\Omega$ satisfies an exterior cone condition
at every point $x\in\bo$, then $\Omega$ is game-regular.
\item[ iii)] If $n=2$ and $\Omega$ is simply connected, then $\Omega$ is
game-regular.
\end{itemize}
\end{thm}

\begin{thm}{\rm(Peres-Sheffield \cite{ps})}
\label{thm:regularandcontinuousconvergence} Let $\Omega\subset\R^n$
be a bounded game-regular domain and $f$
be a continuous function on $\bo$.  Then as $\epsilon \rightarrow
0$, the game values $u^\epsilon_1$ and $u^\epsilon_2$ converge
uniformly to the unique $p$-harmonic function $u$ that extends
continuously to $f$ on $\partial \Omega$.
\end{thm}

\begin{cor}
\label{cor:gameregularregular} Let $\Omega\subset\R^n$ be a bounded
game-regular domain. Then $\Omega$ is also regular.
\end{cor}

Let us think about a probabilistic meaning of a perturbation point
in terms of tug-of-war with noise. If a boundary point is a
perturbation point, it means that the payoff value at that point
does not affect the game value. Therefore, it is naturally guessed
that a perturbation point should be avoidable with high probability
by one player whatever the other player does. This insight makes us
define the following notion.

\begin{defi}
\label{def:gamepertubation}
 $x_0\in\bo$ is called a {\it game-perturbation point} of
 $\Omega$ if  for every $\delta >0$ and $\eta >0$ there exist a $\delta_0$ and
$\epsilon_0$ such that for every $x\in \Omega\cap B(x_0,\delta_0)$
and $\epsilon < \epsilon_0$, player I has a strategy that guarantees
that an $\epsilon$-step game started at $x$ will terminate at a
point on $\partial \Omega \cap B(x_0,\delta)\setminus
B(x_0,\delta_x)$ with probability at least $1 - \eta$ and some
$\delta_x$ which is a constant depending on $x$ with
$0<\delta_x<\delta$. 
\end{defi}

The following lemma will be very useful to a
game-theoretic proof of the results in this paper.

\begin{lem}
\label{lem:gamevaluegameperturbation} Let $1<p<\infty$ and let
$\Omega\subset\R^n$ be a bounded game-regular domain. Suppose that
$f:\bo\to[0,1]$ is continuous. Let $x\in\Omega$ and $\eta>0$. Then
there exists a $\ep_0>0$ such that for every $\epsilon <
\epsilon_0$, player I has a strategy that guarantees that an
$\epsilon$-step game started at $x$ will terminate at a point on
$\{y\in \bo : f(y)>0\}$ with probability at least $H_f(x) - \eta$.
\end{lem}
\begin{proof}
Theorem \ref{thm:regularandcontinuousconvergence} shows that there
exists a $\ep_0>0$ such that for every $\ep\leq \ep_0$, $u_1^\eps(x)
\geq H_f(x)-\eta/2$. Since $u_1^\eps(x) = \sup_{S_I} \inf_{S_{II}}
\mathbb{E}_{x} [f(x_\tau) \mathbf{ \chi}_{\{\tau<\infty\}}]$, player
I has a strategy which guarantees that $ \inf_{S_{II}}$
$\mathbb{E}_{x} [f(x_\tau) \mathbf{ \chi}_{\{\tau<\infty\}}]\geq
u_1^\eps(x)-\eta/2$. Note that
\begin{eqnarray*}
\mathbb{E}_{x} [f(x_\tau) \mathbf{
\chi}_{\{\tau<\infty\}}]&=&\mathbb{E}_{x} [f(x_{\tau}) \mathbf{
\chi}_{\{\tau<\infty\}}, x_{\tau}\in\{y\in \bo : f(y)>0\}]\\
&\leq & \mathbb P_{x}(x_{\tau}\in\{y\in \bo : f(y)>0\}).
\end{eqnarray*}
Therefore, for any player II's strategy,
$$\mathbb P_{x}(x_{\tau}\in\{y\in \bo : f(y)>0\})\geq
u_1^\eps(x)-\eta/2\geq H_f(x)-\eta.$$\end{proof}

Now we are ready to provide a probabilistic characterization of a
perturbation point by using tug-of-war with noise.

\begin{thm}
\label{thm:gameperturbatoin} Let $1<p<\infty$ and let
$\Omega\subset\R^n$ be a bounded game-regular domain. For
$x_0\in\bo$, the following conditions are equivalent.
\begin{itemize}
\item[i)]
$x_0$ is a perturbation point.
\item[ii)]
$x_0$ is a game-perturbation point.
\end{itemize}
\end{thm}
\begin{proof}
First note that $\Omega$ is also regular by Corollary
\ref{cor:gameregularregular}.

$i)\Rightarrow ii)$ Fix $\delta>0$ and $\eta>0$. Define $f:\bo\to
[0,1]$ as a function such that $f=1$ on $\bo\cap B(x_0,\delta/2)$,
$f=0$ on $\bo\setminus B(x_0,\delta)$, otherwise $f$ is continuous.
By the regularity of $x_0$, there exists a $\delta_0>0$ such that
whenever $x\in \Omega\cap B(x_0,\delta_0)$, $H_f(x)\geq 1-\eta/3$.
Let ${\tilde x}\in \Omega\cap B(x_0,\delta_0)$. Let us construct an
increasing sequence $\{g_n\}$ of lower-semicontinuous functions on
$\bo$ by letting $g_n=0$ on $\bo\cap\overline{B(x_0,\delta_0/n)}$,
otherwise $g_n=f$. Note that $g_n$ is resolutive by Theorem
\ref{thm:semicontinuous-resolutive on the regular}. Proposition
\ref{prop:semidecreasing} shows that $\lim H_{g_n}(\tilde x)=\lim
H_{g}(\tilde x)$ where $g$ is a function on $\bo$ such that $g=f$ on
$\bo\setminus\{x_0\}$ and $g(x_0)=0$. Therefore, there exists a $N$
such that $H_{g_N}(\tilde x)\geq H_g(\tilde x)-\eta/3$. Let
$\delta_x=\delta_0/2N$. Let $h:\bo\to [0,1]$ be a continuous
function such that $h\geq g_N$ on $\bo$, $h=g_{N}$ on $\bo\setminus
\overline{B(x_0,\delta_0/N)}$ and $h=0$ on $\bo\cap
{B(x_0,\delta_0/2N)}$. Since $H_h(\tilde x)\geq H_{g_N}(\tilde x)$,
it follows that $H_h(\tilde x)\geq H_g(\tilde x)-\eta/3$. Note that
$H_g=H_f$ because $x_0$ is a perturbation point. Since $H_f(\tilde
x)\geq 1-\eta/3$, it follows that $H_h(\tilde x)\geq 1-2\eta/3$.
Lemma \ref{lem:gamevaluegameperturbation} shows that player I has a
strategy that guarantees that for some $\ep_0$, an $\epsilon$-step
game started at $\tilde x$ with $\ep\leq \ep_0$ will terminate at a
point on $\{x\in \bo : h(x)>0\}$ with probability at least
$H_h(\tilde x) - \eta/3\geq 1-\eta$. Since $\{x\in \bo :
h(x)>0\}\subset \bo\cap B(x_0,\delta)\setminus B(x_0,\delta_0/2N)$,
the proof is complete.

$ii)\Rightarrow i)$ Let $f\in C(\bo)$ and $g$ be a bounded function
on $\bo$ such that $g=f$ on $\bo\setminus\{x_0\}$. To prove that
$H_f=H_g$, it is enough to show that $\lim_{x\in\Omega\to
x_0}H_f(x)= \lim_{x\in\Omega\to x_0}H_g(x)$ by the comparison
principle and the regularity of $\Omega$. Fix $\eta>0$. Since $f$ is
continuous at $x_0$, there exists $\delta>0$ such that for all
$y\in\bo\cap B(x_0,\delta)$, $|f(y)-f(x_0)|\leq \eta$. Let
$x\in\bo\cap B(x_0,\delta_0)$ and let $M=\sup_{\bo}(|f|+|g|)$. Let
$g_M:\bo\to \R$ be a continuous function such that $g_M=f$ on
$\bo\setminus B(x_0,\delta_x)$, $g_M \leq f $ on $B(x_0,\delta_x)$
and $g_M(x_0)=-M$ where $\delta_x$ is given from the assumption that
$x_0$ is a game-perturbation point . Denote by
$u_{g_M}^{1,\epsilon}(x)$ the game value for player I at $x$ with
the payoff function $g_M$. Since $x_0$ is a game-perturbation point,
player I has a strategy which guarantees that for some $\ep_0>0$,
whenever $\ep\leq \ep_0$, $u_{g_M}^{1,\epsilon}(x)\geq
f(x_0)-M\eta$. Letting $\ep\to 0$, Theorem
\ref{thm:regularandcontinuousconvergence} shows that $H_{g_M}(x)\geq
f(x_0)-M\eta$. Since $H_{g}(x)\geq H_{g_M}(x)$ and $x$ is an
arbitrary point on $\bo\cap B(x_0,\delta_0)$,
$\liminf_{x\in\Omega\to x_0} H_g(x)\geq f(x_0)-M\eta$. Letting
$\eta\to 0$ shows that $\liminf_{x\in\Omega\to x_0} H_g(x)\geq
f(x_0)$. Since $x_0$ is a regular boundary point of $\Omega$,
$\lim_{x\in\Omega\to x_0} H_f(x)=f(x_0)$. Therefore,
$\liminf_{x\in\Omega\to x_0} H_g(x)\geq \lim_{x\in\Omega\to x_0}
H_f(x).$ Similarly,  player II adopting the strategy in i) shows
that $\limsup_{x\in\Omega\to x_0} H_g(x)\leq \lim_{x\in\Omega\to
x_0} H_f(x).$ Therefore, $\lim_{x\in\Omega\to x_0}H_f(x)=
\lim_{x\in\Omega\to x_0}H_g(x)$ and the proof is complete.
\end{proof}

\begin{cor}
\label{cor:gameperturbatoin} Let $1<p\leq n$ and let
$\Omega\subset\R^n$ be a bounded game-regular domain. Then every
$x_0\in\bo$ is a game-perturbation point.
\end{cor}
\begin{proof}
The result follows from Theorem \ref{thm:gameperturbatoin} and
Corollary \ref{cor:plessthanninvariant}.
\end{proof}

\section{Characterization of perturbation points}
\label{sec:main result}In this section, we provide a necessary and
sufficient condition for a perturbation point. As Corollary
\ref{cor:plessthanninvariant} shows, our main concern for a
perturbation point is the case of $p>n$. Together with Theorem
\ref{thm:gameperturbatoin}, the following theorem will be a
cornerstone.
\begin{thm}
\label{thm:gameneg} Let $p>n$ and let $\Omega\subset\R^n$ be a
bounded domain. Let $x_0\in\bo$. Then the following conditions are
equivalent.
\begin{itemize}
\item[i)] $x_0$ is a perturbation point.
    \item[ii)] $x_0$ is a game-perturbation point.
\item[iii)]
There exists $\{x_k\}$ such that for all $k\in \N$, $x_k\neq x_0$,
$x_k\in \R^n\setminus\Omega$ and $\lim_k x_k=x_0$.
\item[vi)] $\omega_p(\{x_0\},\Omega)=0$.
\end{itemize}
\end{thm}
\begin{proof}
We prove our statement by showing  $vi)\Rightarrow iii)$,
$iii)\Rightarrow ii)$, $ii)\Rightarrow i)$, and $i)\Rightarrow vi)$.

$vi)\Rightarrow iii)$ Suppose that iii) is not true. Then there
exists $B(x_0,\delta)\subset\Omega$ with some $\delta>0$. Note that
if $p>n$, any bounded domain in $\R^n$ is game-regular by Theorem
\ref{thm:sufficientconditiongameregular}. Therefore, Theorem
\ref{thm:regularandcontinuousconvergence} implies that $
\lim_{x\in\Omega \to x_0}\omega_p(x;\{x_0\},\Omega)=1$, which
contradicts to $\omega_p(\{x_0\},\Omega)=0$.

$iii)\Rightarrow ii)$ The key idea is using an iteration to find a
game-perturbation strategy for player I. Without loss of generality,
we can assume that $x_0=0$. Therefore, there exists $\{x_k\}$ such
that for all $k\in \N$, $x_k\neq 0$, $x_k\in \R^n\setminus\Omega$
and $\lim_k x_k=0$. Inductively we construct a subsequence of
$\{x_k\}$, $\{y_{k}\}$ such that ${|y_k|}$ is decreasing to $0$ and
$$
\frac{|y_{k}|}{|y_{k+1}|}\leq \frac{|y_{k+1}|}{|y_{k+2}|}
~~\textrm{for all}~~ k\in\N.$$ Suppose that we have $\{y_i: 1\leq
i\leq k+1\}$. Then we choose $y_{k+2}$ among $\{x_k\}$ as
$|y_{k+2}|\leq \frac{|y_{k+1}|^2}{|y_k|}$. This can be done
inductively because $\{x_k\}$ is
converging to $0$ and $x_k\neq 0$ for all $k\in \N$. 
For each $k\in \N$, let $$\Omega_k=\{x\in\R^n:
|y_{k+2}|<|x|<|y_k|\}\setminus \{y_{k+1}\}$$ and define a function
$f_k:\Omega_k \rightarrow [0,1]$ as
$f_k(x)=\omega_p\left({x;\{y_{k+1}\},\Omega_k}\right)$. Let
$\theta_k=\inf_x\{f_k(x): x\in\R^n, |x|=|y_{k+1}|\}$.

We show that as $k\to\infty$, $\theta_k$ is increasing, thereby
$\theta_k\geq c>0$ for all $k\in\N$ with some constant $c$. For
this, note that
\begin{eqnarray*}
 \theta_{k+1}&=&\inf_x\{f_{k+1}(x): x\in\R^n, |x|=|y_{k+2}|\}  \\
   &=& \inf_x\left\{f_{k+1}\left(\frac{|y_{k+2}|}{|y_{k+1}|}x\right):
 x\in\R^n, |x|=|y_{k+1}|\right\}\\
&=& \inf_x\left\{
\omega_p\left({\frac{|y_{k+2}|}{|y_{k+1}|}x;\{y_{k+2}\},\Omega_{k+1}}\right):
x\in\R^n,
|x|=|y_{k+1}|\right\}\\
 &=&\inf_x\left\{
\omega_p\left({x;\Big\{\frac{|y_{k+1}|}{|y_{k+2}|}y_{k+2}\Big\},\Omega_k'}\right):
x\in\R^n, |x|=|y_{k+1}|\right\}
\end{eqnarray*}
where $$\Omega_k'=\left\{x\in\R^n:
\frac{|y_{k+3}||y_{k+1}|}{|y_{k+2}|}<|x|<\frac{|y_{k+1}|^2}{|y_{k+2}|}
\right\}\setminus \Big\{\frac{|y_{k+1}|}{|y_{k+2}|}y_{k+2}\Big\}.$$
Here the last equality is obtained by the radial invariance of
$p$-harmonic functions. In addition, a rotational invariance of
$p$-harmonic functions shows that
$$\inf_x\left\{
\omega_p\left({x;\Big\{\frac{|y_{k+1}|}{|y_{k+2}|}y_{k+2}\Big\},\Omega_k'}\right):
|x|=|y_{k+1}|\right\}=\inf_{x}\left\{
\omega_p\left({x;\{y_{k+1}\},\tilde \Omega_k}\right):
|x|=|y_{k+1}|\right\}$$ where $$ \tilde\Omega_n=\left\{x\in\R^n:
\frac{|y_{k+3}||y_{k+1}|}{|y_{k+2}|}<|x|<\frac{|y_{k+1}|^2}{|y_{k+2}|}
\right\}\setminus \{y_{k+1}\}.$$ Therefore, it follows that
\begin{equation}
\label{eqn:theta} \theta_{k+1}=\inf_x\left\{
\omega_p\left({x;\{y_{k+1}\},\tilde \Omega_k}\right): x\in\R^n,
|x|=|y_{k+1}|\right\}.
\end{equation}
Since
$$\frac{|y_{k+3}||y_{k+1}|}{|y_{k+2}|}\leq |y_{k+2}|
 ~~\text{and}~~ \frac{|y_{k+1}|^2}{|y_{k+2}|}\geq |y_k|
,$$ we have that $\Omega_k\subset\tilde\Omega_k$. Since
$\{y_{k+1}\}\subset\partial\Omega_k\cap\partial\tilde \Omega_k$,
Proposition \ref{prop:phamrnoicmeasure} shows that $$
\omega_p\left({x;\{y_{k+1}\},\tilde \Omega_k}\right) \geq
\omega_p\Big({x;\{y_{k+1}\},\Omega_k}\Big)=f_k(x).$$ It follows from
\eqref{eqn:theta} that $\theta_{k+1}\geq\theta_{k}$ for all
$k\in\N$. Moreover, the minimum principle and the regularity of
${y_{k+1}}$ (recall that if $p>n$, any domain in $\R^n$ is regular)
shows that $\theta_1>0$. Therefore, $\theta_k\geq\theta_1>0$ for all
$k\in\N$.

Now we are ready to give a ``game-perturbation strategy" for player
I. First note that when $p>n$, every bounded domain in $\R^n$ is
game-regular by Theorem \ref{thm:sufficientconditiongameregular}.
Fix $\eta>0$ and $\delta>0$. We can find $i,j\in \N$ such that
$(1-\theta_1/2)^i<\eta$ and $|y_j|<\delta$. Let
$\delta_0=|y_{i+j}|$. Let $x_0\in \Omega\cap B(0,\delta_0)$. Since
$|y_k|$ is decreasing to $0$, we can find some $N\in \N$ such that
$x_0\in B(0,|y_{i+j+N-1}|)\setminus {B(0,|y_{i+j+N}|)}$. The
strategy for player I is the following; Let $x_0$ be an initial
point and let
$c=\omega_p\left({x_0;\{y_{i+j+N}\},\Omega_{i+j+N-1}}\right)$. By
the minimum principle, $c>0$. Since $\Omega_{i+j+N-1}$ is
game-regular and
$\omega_p\left({x;\{y_{i+j+N}\},\Omega_{i+j+N-1}}\right)\in
C(\overline{\Omega_{i+j+N-1}})$ is $p$-harmonic, Lemma
\ref{lem:gamevaluegameperturbation}
 shows that player I has a strategy to guarantee that a sufficiently small $\epsilon$-step game
position will arrive at $y_{i+j+N}$ before hitting $\partial
\Omega_{i+j+N-1}\setminus\{y_{i+j+N}\}$ with probability at least
$c/2$. Note that since $y_{i+j+N}\in \R^n\setminus\Omega$, the game
will terminate no later than the game position reaches $y_{i+j+N}$.
Assume that the game position enters $B(0,|y_{i+j+N-1}|)$ before
reaching $y_{i+j+N}$. Then, again by Lemma
\ref{lem:gamevaluegameperturbation} with $f=\chi_{\{y_{i+j+N-1}\}}$
on $\Omega_{i+j+N-2}$, player I can arrange to reach $y_{i+j+N-1}$
before hitting $\partial \Omega_{i+j+N-2}\setminus\{y_{i+j+N-1}\}$
with probability at least $\theta_1/2>0$. Now we iterate this
argument. Whenever the game position enters $B(0,|y_{k+1}|)$ with
some $k\in\N$ before the game terminates, player I adopts a strategy
given by Lemma \ref{lem:gamevaluegameperturbation} with
$f=\chi_{\{y_{k+1}\}}$ on $\Omega_{k}$. Therefore, iterating the
above argument $i$ times shows that player I has a strategy that
guarantees that a sufficiently small $\epsilon$-step game started at
$x_0\in \Omega$ with $|x_0|<\delta_0$ will terminate at a point on $
\{y_k: j\leq k\leq 2i+j+N\}$ with probability at least
$1-(1-c/2)(1-\theta_1/2)^i>1-\eta$. Since $ \{y_k: j\leq k\leq
2i+j+N\}\subset B(0,\delta)\setminus B(0,|y_{2i+j+N+1}|)$, the proof
is complete.

$ii)\Rightarrow i)$ This is a part of the results in Theorem
\ref{thm:gameperturbatoin}.

 $i)\Rightarrow iv)$
This is the general property of a perturbation point. Let $f=0$ and
$g=\chi_{\{0\}}$. Then the result follows.
\end{proof}

As an immediate result, we answer Bj\"orn's problem affirmatively.

\begin{cor}
Let $1<p<\infty$ and let $\Omega\subset\R^n$ be a bounded domain.
Suppose that $x_0\in\bo$ is not an isolated boundary point. Then
$x_0$ is a perturbation point. In particular,
$\omega_p(\{x_0\},\Omega)=0$.
\end{cor}

Theorem \ref{thm:gameneg} gives a necessary and sufficient condition
for a perturbation point in terms of $p$-harmonic measure.

\begin{thm}
\label{thm:pharmoniczeroequivinvariant} Let $1<p<\infty$ and let
$\Omega\subset\R^n$ be a bounded domain. $x_0\in\bo$ is a
perturbation point if and only if $\omega_p(\{x_0\},\Omega)=0$.
\end{thm}
\begin{proof}
When $p>n$, the result follows from Theorem
\ref{thm:pharmoniczeroequivinvariant}. Assume that $1<p\leq n$.  As
Corollary \ref{cor:plessthanninvariant} shows, every boundary point
is a perturbation point. Therefore, we only need to show that
$\omega_p(\{x_0\},\Omega)=0$. However, when $1<p\leq n$, every
single point is of $p$-capacity zero and
$\omega_p(\{x_0\},\Omega)=0$ follows from Theorem
\ref{thm:absoluteequalpcapacity}.
\end{proof}

In addition, when $\Omega$ is game-regular, we have the following.
\begin{thm}
\label{thm:gameperturbation2} Let $1<p<\infty$ and let
$\Omega\subset\R^n$ be a bounded game-regular domain. For
$x_0\in\bo$, the following conditions are equivalent.
\begin{itemize}
\item[i)] $x_0\in\bo$ is a perturbation point.
\item[ii)] $x_0\in\bo$ is a game-perturbation point.
\item[iii)] $\omega_p(\{x_0\},\Omega)=0$.
\end{itemize}
\end{thm}
\begin{proof}
The result follows from Theorem \ref{thm:gameperturbatoin} and
Theorem \ref{thm:pharmoniczeroequivinvariant}.
\end{proof}

As other important consequence of Theorem \ref{thm:gameneg}, we show
the locality of a perturbation point, which is not obvious from the
definition of a perturbation point.

\begin{thm}
\label{thm:invariantlocal} Let $1<p<\infty$ and let
$\Omega_1,\Omega_2\subset\R^n$ be bounded domains. Let $x_0\in
\partial\Omega_1\cap\partial\Omega_2$. Suppose that there exists an open
neighborhood $U$ of $x_0$ such that $U\cap\Omega_1=U\cap\Omega_2$.
Then $x_0$ is a perturbation point of $\Omega_1$ if and only if
$x_0$ is a perturbation point of $\Omega_2$.
\end{thm}
\begin{proof}
By Corollary \ref{cor:plessthanninvariant}, the case of $p>n$ is of
our only concern. In that case, the result follows from ii) in
Theorem \ref{thm:gameneg}.
\end{proof}

\section{Main results for perturbation sets and $p$-harmonic measures}
In this section, we give a necessary and sufficient condition for a
boundary perturbation problem when $f\in C(\bo)$ and $E$ is
countable. As we will see, it also characterize a structure of a
countable set of $p$-harmonic measure zero. Theorem
\ref{thm:perturbationcountablecase} is crucial. Before giving the
result, we introduce two notions. First, we generalize the notion of
a perturbation point to a set.

\begin{defi}
\label{def:perturbationset} Let $\Omega\subset\R^n$ be a bounded
domain. $E\subset \bo$ is called {\it a perturbation set (of
$\Omega$)}; whenever $f\in C(\bo)$ and a bounded function $g$ on
$\bo$ such that $g=f$ on $\bo\setminus E$, $g$ is resolutive and $
H_g=H_f$.
\end{defi}

We can observe that if $E\subset\bo$ is a perturbation set of
$\Omega$, then every $x\in E$ is a perturbation point of $\Omega$
and $\omega_p(E,\Omega)=0$ by letting $f=0$ and $g=\chi_{E}$.
Theorem \ref{thm:bjornpcapcityzero} shows that if $E\subset\bo$ is
of absolute $p$-harmonic measure zero(or equivalently of
$p$-capacity zero), then $E$ is a perturbation set. The following
definition gives a notion which is similar to a set of absolute
$p$-harmonic measure zero.

\begin{defi}
\label{def:omegaabsolutezero} Let $\Omega\subset\R^n$ be a bounded
domain. We say that $E\subset\bo$ is of {\it $\Omega$-absolute
$p$-harmonic measure zero} if
$\omega_p(E\cap\partial\tilde\Omega,\tilde\Omega)=0$ for all bounded
domains $\tilde\Omega$ such that $\tilde\Omega\cap U =\Omega\cap U$
for some open neighborhood $U$ of $E$.
\end{defi}

The following theorem shows a link between a perturbation set and a
set of $p$-harmonic measure zero as well as characterizes a set of
$p$-harmonic measure zero.
\begin{thm}
\label{thm:perturbationcountablecase} Let $\Omega\subset\R^n$ be a
bounded domain and $E\subset\bo$ be a countable set. When
$1<p<\infty$, the following conditions are equivalent.
\begin{itemize}
\item [i)] Every $x\in E$ is a perturbation point of $\Omega$.
\item [ii)] $E$ is a perturbation set of $\Omega$.
\item [iii)] Whenever $\tilde\Omega$ is a bounded domain such that
 $\tilde\Omega\cap U =\Omega\cap U$
for some open neighborhood $U$ of $E$, $E$ is a perturbation set of
$\tilde\Omega$.
\item [iv)] For all $x\in E$, $\omega_p(\{x\},\Omega)=0$.
\item [v)] $\omega_p(E,\Omega)=0$.
\item [vi)] $E$ is of $\Omega$-absolute $p$-harmonic measure zero.
\end{itemize}
Furthermore, when $p>n$, the following conditions are also
equivalent to $i)\sim vi)$.
\begin{itemize}
\item[vii)] Every $x\in E$ is a game-perturbation point of $\Omega$.
\item[viii)] Every $x\in E$ is not an isolated boundary point.
\end{itemize}
\end{thm}
\begin{proof}
Let $1<p<\infty$. To show the equivalence of $i)\sim vi)$, note that
it follows from the definitions that $iii)\Rightarrow ii)\Rightarrow
i)$ and $iii)\Rightarrow vi)\Rightarrow v)\Rightarrow iv)$. Since
Theorem \ref{thm:pharmoniczeroequivinvariant} shows $iv)\Rightarrow
i)$, it is only need to show that $i)\Rightarrow iii)$. Assume $i)$.
Theorem \ref{thm:invariantlocal} shows that every $x\in E$ is also a
perturbation point of $\tilde\Omega$, therefore Theorem
\ref{thm:pmeasureneq} implies that $E$ is a perturbation set of
$\tilde\Omega$. When $p>n$, the equivalence of $i)\sim viii)$
follows from Theorem \ref{thm:gameneg}.
\end{proof}
\remark Note that when $1<p\leq n$, $i)\sim vi)$ are all true.
Therefore, Theorem \ref{thm:perturbationcountablecase} is of special
interest when $p>n$. $i)\sim iii)$ is for a boundary perturbation
problem and $iv)\sim vi)$ is for $p$-harmonic measure.
$i)\Leftrightarrow iv)$ is the repetition of Theorem
\ref{thm:pharmoniczeroequivinvariant}. $ii)\Leftrightarrow v)$ is a
generalization of Theorem \ref{thm:pharmoniczeroequivinvariant}.
Both $iii)$ and $vi)$ show the locality of a boundary perturbation
problem and $p$-harmonic measure when $E$ is countable. Compare
$vi)$ to $iii)$ in Proposition \ref{prop:phamrnoicmeasure}. When
$p>n$, $viii)$ provides a geometric criterion to show that $E$ is a
perturbation set of $\Omega$ or equivalently $\omega_p(E,\Omega)=0$.
\begin{thm}
\label{thm:countablesubadditive} Let $1<p<\infty$ and let
$\Omega\subset\R^n$ be a bounded domain. For each $k\in \N$, assume
that $E_k\subset\bo$ is a countable set and
$\omega_p(E_k,\Omega)=0$.
 Then
$$\omega_p(\cup_k E_k,\Omega)=0.$$
\end{thm}
\begin{proof}
Let $x\in \cup_k E_k$.  Since $\omega_p(\cup_k E_k,\Omega)=0$,
$\omega_p( \{x\},\Omega)=0$. Therefore the result follows from
$iv)\Leftrightarrow v)$ in Theorem
\ref{thm:perturbationcountablecase}.
\end{proof}

Next we give an invariance result for $p$-harmonic measure.
\begin{thm}
\label{thm:countablepmeasureinvari} Let $1<p<\infty$ and let
$\Omega\subset\R^n$ be a bounded domain. Suppose that $E\subset \bo$
is a countable set with $\omega_p(E,\Omega)=0$. Then for every
closed set $F\subset\bo$,
$$\omega_p(x;E\cup F,\Omega)=\omega_p(x;F,\Omega) ~\text{for~all}~x\in \Omega.$$
\end{thm}
\begin{proof}
It suffices to show that $\omega_p(x;F,\Omega)\geq\omega_p(x;E\cup
F,\Omega)$. We can approximate $\chi_F$ by a decreasing sequence of
continuous function $\{f_n\}$ such that $\lim_n f_n=\chi_F$ on
$\bo$. Proposition \ref{prop:semidecreasing} shows that
$\lim_nH_{f_n}(x)=\omega_p(x;F,\Omega)$ for all $x\in\Omega$. Note
that $E$ is a perturbation set of $\Omega$ by Theorem
\ref{thm:perturbationcountablecase}, thereby
$H_{f_n}(x)=H_{f_n+\chi_{E}}(x)$. Since $H_{f_n+\chi_{E}}(x)\geq
\omega_p(x;E\cup F,\Omega)$, letting $n\to \infty$ shows that
$\omega_p(x;F,\Omega)\geq\omega_p(x;E\cup F,\Omega)$.
\end{proof}
\remark When $1<p\leq n$, Kurki \cite{kurki} proved a similar
invariance result by assuming that $E$ is a set of $p$-capacity zero
instead of a countable set of $p$-harmonic measure zero. However, as
the author is aware, Theorem \ref{thm:countablepmeasureinvari} is a
first invariance result for $p$-harmonic measure when $p>n$.

At last, we give a partial answer to Open problem
\ref{ques:pharmonicmeasure} in some extreme cases; for any two
closed subsets $E,F\subset\bo$ with
$\omega_p(x;E,\Omega)=\omega_p(x;F,\Omega)=0$, $\omega_p(x;E\cup
F,\Omega)=0$ if $E$ and $F$ are either somewhat  ``heavily"
overlapped or``slightly" overlapped. The latter case is a slight
generalization of Theorem \ref{compact sets}.
\begin{thm}
Let $1<p<\infty$ and let $\Omega\subset\R^n$ be a bounded regular
domain. Let $E,F\subset\bo$ are closed sets of $p$-harmonic measure
zero. Further assume that either $(E\cup F)\setminus (E\cap F)$ is
countable or there exists a closed set $G\subset\bo$ such that
$G\subset F\setminus E$ and ${F\setminus G}$ is countable. Then
$\omega_p(E\cup F,\Omega)=0$.
\end{thm}
\begin{proof}
Since $E\cup F=(E\cap F)\cup\{(E\cup F)\setminus (E\cap F)\}$ and
$E\cap F$ is a closed set of $p$-harmonic measure zero, the result
follows from Theorem \ref{thm:countablepmeasureinvari}. For ii) note
that $E$ and ${G}$ are two disjoint closed sets of $p$-harmonic
zero. Theorem \ref{compact sets} shows $\omega_p(E\cup G,\Omega)=0$.
Since $E\cup F=(E\cup G)\cup (F\setminus G)$ and $F\setminus G$ is a
countable set of $p$-harmonic measure zero, the result follows again
from Theorem \ref{thm:countablepmeasureinvari}.
\end{proof}

\section {Open problems}
\label{sec:other remarks} Let $1<p<\infty$ and let
$\Omega\subset\R^n$ be a bounded domain throughout this section. It
is easy to check that if $E\subset\bo$ is a perturbation set,  $E$
is of $p$-harmonic measure zero. When $E$ is a countable set of
$p$-harmonic measure zero, Theorem
\ref{thm:perturbationcountablecase} shows that the converse is also
true, i.e. if $E$ is of $p$-harmonic measure zero, then $E$ is a
perturbation set. We may wonder whether this is still true when $E$
is not a countable set.
\begin{question}
If $E\subset\bo$ is of $p$-harmonic measure zero, then is $E$ a
perturbation set?
\end{question}

Let us recall that Theorem \ref{thm:absoluteequalpcapacity} and
Theorem \ref{thm:bjornpcapcityzero} show that if $E\subset\bo$ is of
absolute $p$-harmonic measure zero or equivalently of $p$-capacity
zero, then $E$ is a perturbation set of $\Omega$. The converse is
generally not true. However, when $E$ is a countable set, Theorem
\ref{thm:perturbationcountablecase} shows that $E$ is a perturbation
set of $\Omega$ if and only if $E$ is of $\Omega$-absolute
$p$-harmonic measure zero. This fact makes us conjecture the
following question.

\begin{question}
Is it true that $E\subset\bo$ {is a perturbation set} if and only if
 $E$ {is of $\Omega$-absolute $p$-harmonic
measure zero}?
\end{question}

If the answers to the above two problems are yes, we can give an
affirmative answer to the following open problem.

\begin{question}
If $E\subset\bo$ is of $p$-harmonic measure zero, then is $E$ {of
$\Omega$-absolute $p$-harmonic measure zero}?
\end{question}

\bigskip
\noindent{\bf Acknowledgments.} I am grateful to Anders Bj\"orn and
Scott Sheffield for their helpful advice.

\bibliographystyle{alpha}
\bibliography{Sungwook_Kim_arxiv}
\end{document}